\documentclass[preprint]{elsarticle}
\usepackage{hyperref}
\usepackage{amsmath,amsthm}
\usepackage{amssymb}
\usepackage{tikz}
\usepackage{vmargin}

\newtheorem{thm}{Theorem}

\newtheorem{Prop}[thm]{Proposition}
\newtheorem{lemma}[thm]{Lemma}
\newtheorem{corollary}[thm]{Corollary}

\journal{Discrete Mathematics}
\begin{document}
\begin{frontmatter}

\title{Locally identifying colourings for graphs
with given maximum degree\tnoteref{t1}}

 \tnotetext[t1]{This research is
    supported by the ANR Project IDEA {\scriptsize $\bullet$} {ANR-08-EMER-007},  2009-2011.}

\author[bdx]{Florent Foucaud}
\author[tku]{Iiro Honkala}
\author[tku]{Tero Laihonen}
\author[gre]{Aline Parreau}
\ead{aline.parreau@ujf-grenoble.fr}
\author[cat]{Guillem Perarnau}

\address[bdx]{Universit\'e de Bordeaux, LaBRI, 351 
cours de la Lib\'eration, 33405 {Talence}, France}
\address[tku]{Department of Mathematics, University of Turku, 20014 {Turku}, Finland}
\address[gre]{Institut Fourier, 100 rue des Maths, BP 74,
38402 {Saint-Martin d'H\`eres}, France}
\address[cat]{COMBGRAPH, Universitat Polit\`ecnica de Catalunya, Spain}

\begin{abstract}
A proper vertex-colouring of a graph $G$ is said to be locally identifying if for any pair $u$,$v$ of adjacent vertices with distinct closed neighbourhoods, the sets of colours in the closed neighbourhoods of $u$ and $v$ are different. We show that any graph $G$ has a locally identifying colouring with $2\Delta^2-3\Delta+3$ colours, where $\Delta$ is the maximum degree of $G$, answering in a positive way a question asked by Esperet {\em et al}. We also provide similar results for locally identifying colourings which have the property that the colours in the neighbourhood of each vertex
are all different and apply our method to the class of chordal graphs.
\end{abstract}

\end{frontmatter}

\section{Introduction}
Let $G=(V,E)$ be a simple undirected finite graph. Let $c: V \to \mathbb{N}$ be a colouring of the vertices of $G$. For a subset $S$ of $V$, we denote by $c(S)$ the set of colours that appear in $S$: $c(S)=\{c(u)~|~u\in S\}$ and we denote by $N(u)$ (resp. $N[u]$) the open (resp. closed) neighbourhood of $u$: $N(u)=\{v \in V~|~uv\in E\}$ (resp. $N[u)]=N(u)\cup\{u\}$).

The colouring $c$ is a {\em locally identifying colouring} ({\em lid-colouring for short})
 if it is a proper colouring (no two adjacent vertices have the same colour) such that for each pair of
adjacent vertices $u$,$v$ with $N[u]\neq N[v]$, we have $c(N[u]) \neq c(N[v])$. A colour belonging to the symmetric difference of $c(N[u])$ and $c(N[v])$ is said to {\em separate} $u$ from $v$.
An edge $uv$ is said to be {\em bad} if $N[u]\neq N[v]$ and $c(N[u]) = c(N[v])$. So a locally
identifying colouring is a proper vertex colouring without any bad edge.
The {\em locally identifying chromatic number} of $G$, denoted by $\chi_{lid}(G)$, is the minimum number of
colours required in any locally identifying colouring of $G$.

Locally identifying colourings have been introduced in \cite{EGMOP10} and are related to  identifying codes \cite{KCL98,lobs}, distinguishing colourings \cite{BRS03,BS97,CHS96} and locating-colourings \cite{CEHSZ02}.
An open question asked in \cite{EGMOP10} was to know whether one can find a locally identifying colouring of a graph $G$ with $O(\Delta^2)$ colours, where $\Delta$ is the maximum degree of $G$.
Examples using the projective plane provide graphs $G$ with $\chi_{lid}(G)= \Delta^2-\Delta+1$ (see \cite{EGMOP10}). In this
note, we show that we always have $\chi_{lid}(G)\leq 2\Delta^2-3\Delta+3$, answering in a positive way the question
asked in \cite{EGMOP10}. The result is effective: we give a construction for a locally identifying colouring
with $2\Delta^2-3\Delta+3$ colours. This construction can be slightly modified, using
$2\Delta^2-\Delta+1$ colours to provide a locally identifying colouring which has the property that the colours in the neighbourhood of each vertex are all distinct. We finally consider the class of chordal graphs, for which it is conjectured in \cite{EGMOP10} that $\chi_{lid}(G)\leq 2\chi(G)$, for any chordal graph $G$. We give a bound for this class in terms of $\Delta$ and $\chi$, in the direction of the previous conjecture.
For terminology and notations of graph theory, we refer to the book \cite{BM08}.

\section{Upper bound in terms of the maximum degree}

The following lemma shows that, given a locally identifying colouring,
we can change the colour of a single vertex in a number of ways,
without  sacrificing the property that the colouring
is locally identifying.

\begin{lemma}[Recolouring Lemma]\label{lem:rec}
Let $G$ be a graph with maximum degree $\Delta\geq 3$. Let $v$ be a vertex of degree $d$.
Assume that $G$ has a locally identifying colouring $c$ with strictly more than $2d(\Delta-1)$ colours. Then, there is a list $L$ of colours of size at most $2d(\Delta-1)$ such that if we colour $v$ with a colour not in $L$, the colouring remains locally identifying.
\end{lemma}

\begin{proof}
Let $v_1,...,v_d$ be the neighbours of $v$.
For each vertex $v_i$, let $u_{i,1},...,u_{i,s_i}$ be the neighbours of $v_i$ that are not
neighbours of $v$, see Figure~\ref{fig:notation}. For $1\leq i \leq d$, we construct a list $L_i$ of colours with at most $2(\Delta-1)$ colours. We first put $c(v_i)$ in $L_i$.

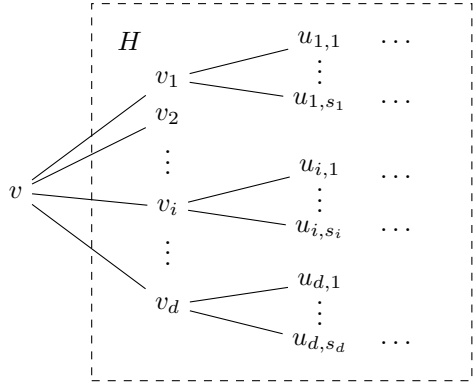
\begin{figure}[ht]
\begin{center}
\begin{tikzpicture}
\node(v) at (-1,0.5) {$v$};
\node(v1) at (1,2) {$v_1$};
\node(v2) at (1,1.5) {$v_2$};
\node at (1,1) {$\vdots$};
\node(vi) at (1,0.3) {$v_i$};
\node at (1,-0.2) {$\vdots$};
\node(vd) at (1,-1) {$ v_d$};

\node(u11) at (3,2.5) {$u_{1,1}$};
\node at (3,2.2) {$\vdots$};
\node(u12) at (3,1.7) {$u_{1,s_1}$};

\node(ui1) at (3,0.8) {$u_{i,1}$};
\node at (3,0.5) {$\vdots$};
\node(ui2) at (3,0) {$u_{i,s_i}$};

\node(ud1) at (3,-0.7) {$u_{d,1}$};
\node at (3,-1) {$\vdots$};
\node(ud2) at (3,-1.5) {$u_{d,s_d}$};

\node(u1) at (4,2.5) {$\ldots$};
\node(u1) at (4,1.7) {$\ldots$};
\node(u1) at (4,0) {$\ldots$};
\node(u1) at (4,-1.5) {$\ldots$};
\node(u1) at (4,0.7) {$\ldots$};

\draw[-] (v)--(v1);
\draw[-] (v)--(v2);
\draw[-] (v)--(vi);
\draw[-] (v)--(vd);
\draw[-] (v1)--(u11);
\draw[-] (v1)--(u12);
\draw[-] (vi)--(ui1);
\draw[-] (vi)--(ui2);
\draw[-] (vd)--(ud1);
\draw[-] (vd)--(ud2);
\draw[dashed] (5,3) rectangle (0,-2);
\node at (0.5,2.5) {$H$};
\end{tikzpicture}
\end{center}
\caption{\label{fig:notation} Neighbourhood of a vertex $v$}
\end{figure}

If there is a vertex $u_{i,j}$ such that $c(N[u_{i,j}])=c(N[v_i]\setminus \{v\})$, (in other words, $u_{i,j}$ is separated from $u_i$ only by the colour of $v$), then we say that $v_i$ is {\em of type $A$} and
we add to $L_i$ all the colours of $c(N[v_i]\setminus \{v\})$. We add at this point at most
$\Delta-1$ colours because $c(v_i)$ is already in $L_i$.
Then, for all vertices $u_{i,j'}$ such that $c(N[u_{i,j'}])\setminus c(N[v_i]\setminus \{v\})$ is
not empty, we add an arbitrary colour of  $c(N[u_{i,j'}])\setminus c(N[v_i]\setminus \{v\})$ to $L_i$. In
this step, we add at most $\Delta-2$ colours because $j'\neq j$. Therefore, in the end $|L_i|\leq
2(\Delta-1)$.

Otherwise, we say that $v_i$ is {\em of type $B$} and for each neighbour $u\neq v$ of $v_i$, if $c(N[u]\setminus \{v\})\setminus c(N[v_i]\setminus \{v\})$ is not empty,
we add one colour of this set to $L_i$. Note that $u$ can be some other vertex $v_j$ or some vertex
$u_{i,j}$, but there are at most $\Delta-1$ such vertices. If $c(N[v_i])\setminus c(N[v])$ is not empty, we add one colour of this set to $L_i$. In the end,  $|L_i|\leq \Delta+1$.

Let $L=\cup_{i=1,\ldots,d} L_i$. Because $\Delta \geq 3$, $|L|\leq 2d(\Delta -1)$.
We define a new colouring $c'$ of $G$ by just giving to $v$ any colour not in $L$. We
will prove that $c'$ is locally identifying.
First, $c'$ is a proper colouring because $L$ contains all the colours $c(v_1)$,...,$c(v_d)$.
Let now $x,y$ be a pair of adjacent vertices, with $N[x]\neq N[y]$. We will show that $c'(N[x])\neq
c'(N[y])$. If neither $x$ nor $y$ are in $N[v]$ then the colours in their respective neighbourhoods
did not change and we have $c'(N[x])\neq c'(N[y])$.
So we can assume, without loss of generality, that $x=v_1$.

Suppose first that $y=v$. Let us also assume that there exists a colour 
$c_0 \in c(N[v_1])\setminus c(N[v])$, then $c_0\neq c(v)$. If $v_1$ is
of type $A$, then $c_0\in L$ and $c_0 \in c'(N[v_1])\setminus c'(N[v])$. Hence $c_0$ is still separating
$v$ from $v_1$. 
If $v_1$ is of type $B$, then at least one colour of  
$c(N[v_1])\setminus c(N[v])$ is in $L$ (not necessarily $c_0$), and is separating $v_1$ from $v$.
Now, assume that $c(N[v_1])\setminus c(N[v])$ is empty. Because $c(N[v_1])\neq c(N[v])$, there is a colour $c_0 \in c(N[v])\setminus c(N[v_1])$. We have $c_0\neq c(v)$, so $c_0\in L$ and $c_0$ is still separating $v$ from $v_1$.

Assume now that $y=v_j$, with $j\neq 1$ and $v_1$,$v_j$ adjacent. Then without loss of generality we can assume that there exists a colour $c_0$ in $c(N[v_1])$ that does not appear in $c(N[v_j])$. Necessarily, $c_0\neq c(v)$ and so $c_0 \in
c(N[v_1]\setminus\{v\})\setminus c(N[v_j]\setminus\{v\})$. If $v_1$ is of type $A$, then $c_0$ is not the new
colour $c'(v)$ of $v$ because $c_0\in L$. Otherwise, $v_1$ is of type $B$. 
If $c(N[v_j]\setminus \{v\})\setminus c(N[v_1]\setminus\{v\})$ is nonempty, then there is one
colour  in $L$ of $c(N[v_j]\setminus \{v\})\setminus c(N[v_1]\setminus\{v\})$ that is separating $v_j$ from $v_1$. Otherwise, if $v_j$ is of type $A$, we are done using a similar argument. If $v_j$ is of type $B$, then one colour of $c(N[v_1]\setminus\{v\})\setminus c(N[v_j]\setminus\{v\})$ is in $L$ and is separating the two vertices.

Finally, we can assume without loss of generality that $y=u_{1,1}$. If
$c(N[u_{1,1}])=c(N[v_1]\setminus \{v\})$, then $v_1$ is of type $A$ and so $c(N[u_{1,1}])\subseteq
L$. Hence, the new colour of $v$, $c'(v)$, is not in $c(N[u_{1,1}])$ and is separating $v_1$ from
$u_{1,1}$. If there is a colour in $c(N[v_1]\setminus \{v\})\setminus c(N[u_{1,1}])$, then it is
still separating $v_1$ from $u_{1,1}$. Otherwise $c(N[u_{1,1}])\setminus
c(N[v_1]\setminus \{v\})$ is necessarily nonempty, and so there is a colour of $L$ that is separating $u_{1,1}$
from $v_1$.
\end{proof}

Let $d$ be an integer. A graph $G$ is {\em $d$-degenerate} if each subgraph of $G$ has a vertex of degree at most $d$ (see \cite{LW70} for reference). 

\begin{Prop}\label{prop:coldeg}
Let $G$ be a $d$-degenerate connected graph with maximum degree $\Delta\geq 3$ and $d<\Delta$. Then:
$$\chi_{lid}(G)\leq 2(\Delta-1)^2+d.$$
\end{Prop}

\begin{proof}
Let $\Delta$ be fixed. We prove the claim by induction on the number of vertices in $G$. In the conditions of lemma, $2(\Delta-1)^2+d\geq 9$ so the claim
is true for graphs with less than nine vertices. Assume that the claim is true for every $d$-degenerate
graph with fewer than $n$ vertices and maximum degree at most $\Delta$.
Let $G$ be a $d$-degenerate graph with $n$ vertices and maximum degree at most $\Delta$.

Let $v$ be one vertex of minimum possible degree $t\leq d$ in $G$. Let $H=G\setminus\{v\}$.
The graph $H$ is also $d$-degenerate and by the induction hypothesis, there is a locally identifying colouring $c$ of $H$ with $2(\Delta-1)^2+d$ colours.
As in the previous lemma, we denote the $t$ neighbours of $v$ by $v_1$,...,$v_t$, (if $v$ has no
neighbours the claim is trivial) and for each $i\in\{1,...,t\}$, we denote by
$u_{i,1}$,...,$u_{i,s_i}$ the neighbours of $v_i$ that are not neighbours of $v$ (see Figure~\ref{fig:notation}). We construct a list $L'$ of size at most $d$ containing, for each $i\in
\{1,..,t\}$, the colour $c(u_{i,1})$ if $u_{i,1}$ exists.
Each vertex $v_i$ has degree  at most $\Delta-1$ in $H$. Using Lemma \ref{lem:rec}, we can choose for each vertex $v_i$ a colour that is not in $L'$: indeed, there are $2(\Delta-1)^2$ forbidden
colours from the lemma applied to $v_i$ and $d$ colours in $L'$, but the colours $c(u_{i,1})$, if $u_{i,1}$ exists, is
counted twice, so there are at most $2(\Delta-1)^2+d-1$ forbidden colours and at least one colour is free.

We can now assume that $c$ is a locally identifying colouring of $H$ such that no vertex $v_i$ has a colour in
$L'$. We now assign to $v$ a new colour, $c(v)$, never used in $c$ originally. We will prove that this colouring of
$G$ is locally identifying. It is clearly a proper colouring, and the only pair of adjacent vertices
that is not clearly separated is the pair $\{v,v_i\}$. If $v_i$ has a neighbour $u_{i,1}$, then
$c(u_{i,1})\in L'$ and $c(N[v])$ is not containing $c(u_{i,1})$, so they are separated.
Otherwise, we have $N[v_i]\subseteq N[v]$. But $v$ has minimum degree, so necessarily 
$N[v_i]= N[v]$ and the two vertices do not need to be separated.

Finally, we obtain a lid-colouring of $G$ using $2(\Delta-1)^2+d+1$ colours. Let $\alpha$ be the new colour of $v$. Recall that $\alpha$ is used only once. We want to apply the recolouring lemma to $v$ and avoid the colour $\alpha$. So $2(\Delta-1)^2+d-2d(\Delta-1)$ has to be greater than $0$ (the number of colours used in the colouring minus the number of forbidden colours of the recolouring lemma applied to $v$ minus the colour $\alpha$). This is a decreasing function linear in $d$ whose minimum is $\Delta-1$ (when $d=\Delta-1$) which is striclty greater than $0$, as $\Delta\geq 3$. This leads to a lid-colouring of the whole graph $G$ using at most $2(\Delta-1)^2+d$ colours.
\end{proof}

\begin{corollary}
Let $G$ be a graph with maximum degree $\Delta\geq 3$. Then: $$\chi_{lid}(G)\leq
2\Delta^2-3\Delta+3.$$
\end{corollary}

\begin{proof}
If $G$ is not connected, we colour the components independently. So we can assume that $G$ is connected.

If $G$ is not $\Delta$-regular, then $G$ is $(\Delta-1)$-degenerate. Indeed, if we take any proper subset $V'$ of vertices, and consider the subgraph induced by $V'$, then if
every vertex in this induced subgraph had degree $\Delta$, there would be
no edges between $V'$ and $V(G)\setminus V'$, and therefore $G$ would not be connected.
So in this case we can directly apply Proposition~\ref{prop:coldeg} and the result is clear.
%

Assume now that $G$ is $\Delta$-regular.
Let $v$ be any vertex of $G$. As before, the graph $G\setminus\{v\}$ is $(\Delta-1)$-degenerate and so, by Proposition~\ref{prop:coldeg}, it has a
lid-colouring with $2\Delta^2-3\Delta+1$ colours.
As in the proof of Proposition~\ref{prop:coldeg}, we can recolour all the neighbours of $v$ in such a way that the colouring remains
locally identifying and such that if a neighbour of $v$ has not the same closed neighbourhood as $v$, it has a neighbour
with a colour different from all the colours of $N(v)$. More precisely, to recolour a neighbour of $u$, Lemma \ref{lem:rec} forbidds $2(\Delta-1)^2=2\Delta^2-4\Delta+2$ colours and we forbidd also the colours of $N(v)\setminus \{v,u\}$, i.e. at most $\Delta-1$ more colours. So in total, there are
$2\Delta^2-3\Delta+1$ forbidden colours. To apply Lemma~\ref{lem:rec},  we have to start with a lid-colouring using $2\Delta^2-3\Delta+2$ colours which corresponds to the lid-colouring given by Proposition~\ref{prop:coldeg} plus an unused color. We finally get a lid-colouring of $G\setminus\{v\}$ with $2\Delta^2-3\Delta+2$ colours and the required properties on the neighbours of $v$. 
Then we assign a completely new colour to $v$ and as in the proof of Proposition~\ref{prop:coldeg}, we can show that the colouring is locally identifying, leading to a locally identifying colouring of the whole graph $G$ with $2\Delta^2-3\Delta+3$ colours .
\end{proof}

We now study the case $\Delta=2$:

\begin{Prop}
Let $n\geq 4$ be an integer. Let $\mathcal C_n$ be the cycle of order $n$. Then:
\begin{itemize}
\item $\chi_{lid}(\mathcal C_n)=3$ if $n\equiv 0 \bmod 4$,
\item $\chi_{lid}(\mathcal C_n)=5$ if $n=5$ or $7$,
\item $\chi_{lid}(\mathcal C_n)=4$ otherwise.
\end{itemize}

As a consequence, any graph with maximum degree 2 has a locally identifying colouring with five colours.
\end{Prop}

\begin{proof}

Let $v_0$,..., $v_{n-1}$ be the vertices of $\mathcal C_n$.
We clearly have $\chi_{lid}(\mathcal C_n)\geq 3$.

We colour $\mathcal C_n$ with four colours using  the following family of sequences described by the
following word: $$[124341232][42](1232)^*$$ for $n\geq 4$ and $n\neq 5,7$. A sequence in bracket,
$[M]$, means that we can take or not take the sequence $M$, the sequence $(M)^*$ means that we can repeat
sequence $M$ as many times as we need (or not use it at all). One can check that if we colour
vertices of $\mathcal C_n$ with one of the sequences described by the previous word, we obtain a
locally identifying colouring with three colours if $n\equiv 0 \bmod 4$ and with four colours otherwise.

If $n\not \equiv 0\bmod 4$, then there is no locally identifying colouring with three colours. Indeed, if we try to
colour the vertices of $\mathcal C_n$ with three colours there is no choice to do it and we must
colour, without loss of generality: $c(v_i)$ with colour $1$ if $i\equiv 0 \bmod 4$, with colour $2$
if $i\equiv 1,3\bmod 4$ and with colour $3$ if $i\equiv 2 \bmod 4$. But $v_{n-1}$ must have colour
$2$, and $v_{n-2}$ must have colour $3$. Then $n-2\equiv 2 \bmod 4$ and so $n\equiv 0 \bmod 4$, a
contradiction.

A case analysis shows that $\chi_{lid}(\mathcal C_5)=\chi_{lid}(\mathcal C_7)=5$.

For the last part of the proposition, if $G$ has maximum degree $2$ then it is composed of connected components that are cycles or paths. One can easily check that a path has always a locally identifying colouring with four colours, and so we can colour each
connected component  of $G$ independently with at most five colours.
\end{proof}

One can notice that in the locally identifying colourings of the cycle provided in the proof, only three colours are
used an unbounded number of times, whereas the other colours are used at most three times. In some sense, we can say that $\mathcal C_n$ has {\em almost} a locally identifying colouring with three colours.

\section{Strong locally identifying colourings}

In this section, we consider a variation of locally identifying colourings by adding a strong constraint to the definition. Our technique can still be applied in this context in order to get a bound in the same asymptotic order.
We also extend our method to the class of chordal graphs, in relation with a conjecture of \cite{EGMOP10}.

We say that a colouring $c$ is a  {\em strong locally identifying colouring} ({\em slid-colouring} for short)
if it  is a locally identifying colouring and if for each vertex $u$, all the colours in $N[u]$ are different (the colouring
is locally injective).
In other words, a slid-colouring is a proper distance-two vertex-colouring (two vertices at distance at most $2$ from each other have different colours) without bad edges.

We denote by $\chi_{slid}(G)$ the minimum number of colours required in any slid-colouring of $G$. Clearly, $\chi_{lid}(G)\leq \chi_{slid}(G)$ and $\chi_2(G)\leq \chi_{slid}(G)$, where $\chi_2(G)$ denotes the minimum number of colours in a distance-two vertex-colouring of $G$. In a graph $G$ with maximum degree $\Delta$, each vertex has at most $\Delta^2$ vertices at distance at most 2, and so $\chi_2(G)\leq \Delta^2+1$. 

We will adapt the proof of the previous section to show that
merging  the locally-identifying constraint with the distance-two colouring constraint does not increase the asymptotic order of both bounds.


\begin{lemma}[Recolouring lemma 2]\label{lem:rec2}
Let $v$ be a vertex of degree $d_1$ of a graph $G$. Assume that $v$ has $d_2$ vertices at distance
exactly $2$ and let $c$ be a slid-colouring $c$ with strictly more than $d_1+2d_2$ colours. Then, there is a list $L$ of colours of size at most $d_1+2d_2$ such that if  we colour $v$ with a colour not in $L$, the colouring remains a slid-colouring.
\end{lemma}

\begin{proof}
In this case we also need to put in $L$ all the colours of vertices at distance $2$ of $v$.
 We keep the same notations as before and we construct $L$ as follows:
\begin{enumerate}
\item For each vertex $v_i$, add colour $c(v_i)$ to $L$ (at most $d_1$ colours are added).
\item For each vertex $u_{i,j}$, add colour $c(u_{i,j})$ to $L$  (at most $d_2$ colours are added).
\item For each vertex $u_{i,j}$, if there is some colour in $c(N[u_{i,j}])$, but not already in
$L$, then add one of them to $L$. At this step we add at most $d_2$ colours.
\end{enumerate}

In the end, $L$ contains at most $d_1+2d_2$ colours. If we consider a new colouring $c'$ from $c$
where we colour $v$ with a colour not in $L$, then clearly $c'$ is still a distance-two colouring. Furthermore, no edge becomes bad. Indeed, the only edges that could become bad would
be of the form $v_iu_{i,j}$. There are two cases depending on whether $c(N[u_{i,j}])$ is included in $c(N[v_i])$ or not.  

If $c(N[u_{i,j}])\subset c(N[v_i])$, then there is a colour $c_0\in c(N[v_i])\setminus
c(N[u_{i,j}])$. If $c_0$ was the colour $c(v)$, because $c'(v)\notin c(N[u_{i,j}])\subset c(N[v_i])$, we have $c'(v)\in
c'(N[v_i])\setminus c'(N[u_{i,j}])$. Otherwise, we still have $c_0\in c'(N[v_i])\setminus
c'(N[u_{i,j}])$.

Otherwise  $c(N[u_{i,j}])\setminus c(N[v_i])$ is not empty and during the construction of $L$, one colour of $c(N[u_{i,j}])\setminus c(N[v_i])$ has been added to $L$, still separating $v_i$ from $u_{i,j}$.
\end{proof}

Note that Recolouring lemma $2$ can also be applied in the locally identifying colouring case, leading in some cases to a better bound than the one of Recolouring lemma~\ref{lem:rec}.
\bigskip

As before, we can use this lemma to construct a slid-colouring of a graph $G$  by induction. We
first colour $G$ when it is $d$-degenerate:

\begin{Prop}
Let $G$ be a $d$-degenerate graph with maximum degree $\Delta\geq 2$ and $d<\Delta$. Then:
$$\chi_{slid}(G)\leq (\Delta-1)(2\Delta-1)+2d-1.$$
\end{Prop}

\begin{proof}
The idea of the proof is similar to the proof of Proposition~\ref{prop:coldeg}. 
We construct the colouring by induction. We choose a vertex $v$ with minimum possible degree $t\leq d$. Then
$G\setminus \{v\}$ has a slid-colouring  with $(\Delta-1)(2\Delta-1)+2d-1$ colours. 
For each neighbour $v_i$ of $v$ which has a neighbour $u_{i,1}$ at distance $2$ of $v$, we put the colour
of $u_{i,1}$ in a list $L'$.
We recolour each neighbour $v_i$ of $v$ in such a way that all the neighbours of $v$ have different
colours and none of them has a colour in $L'$.
To recolour $v_i$, there are  at most $(\Delta-1)(2\Delta-1)$ forbidden colours from the lemma,
 at most $d-1$ colours from the other neighbours of $v$ and at most $d-1$ forbidden colours from $L'$ (if $v_i$ has a
neighbour $u_{1,1}$ then the colour of $u_{1,1}$ is already forbidden for $v_i$ in the lemma). Therefore, at
most  $(\Delta-1)(2\Delta-1)+2d-2$ are forbidden but we have $(\Delta-1)(2\Delta-1)+2d-1$ colours,
so at least one colour is free.
After that, we colour $v$ with a completely new colour, obtaining a slid-colouring with
$(\Delta-1)(2\Delta-1)+2d$ colours, and by Lemma~\ref{lem:rec2}, we can change the colour of $v$ to a colour already used. We thus obtain a slid-colouring with $(\Delta-1)(2\Delta-1)+2d-1$ colours (for this last step, at least two colours are free so one of them is not the colour of $v$ and we can change the colour of $v$ to a colour already used).
\end{proof}

\begin{corollary}
Let $G$ be a graph with maximum degree $\Delta$. Then:
$$\chi_{slid}(G)\leq  2\Delta^2-\Delta+1.$$
\end{corollary}

\begin{proof}
As before, we can assume that $G$ is connected. If $G$ has a vertex of degree $d<\Delta$, then it is $(\Delta-1)$-degenerate and we have
$\chi_{slid}(G)\leq  2\Delta^2-\Delta -2$. Otherwise, $G$ is $\Delta$-regular. If we remove one
vertex $v$, then $G\setminus\{v\}$ is $(\Delta-1)$-degenerate and there is a slid-colouring with
$2\Delta^2-\Delta -2$ colours. We recolour the neighbours of $v$ as before, but here there are $\Delta$
neighbours so we will need $(\Delta-1)(2\Delta-1)+2\Delta-2+1=2\Delta^2-\Delta$ colours. We complete
the colouring by giving a completely new colour to $v$, thus obtaining a slid-colouring with
$2\Delta^2-\Delta+1$ colours.
\end{proof}

We then consider the case of the cycle:

\begin{Prop}
Let $n\geq 4$ be an integer. Let $\mathcal C_n$ be the cycle of order $n$. Then:
\begin{itemize}
\item $\chi_{slid}(\mathcal C_n)=4$ if $n\equiv 0 \bmod 4$,
\item $\chi_{slid}(\mathcal C_n)=6$ if $n=6$ or $11$,
\item $\chi_{slid}(\mathcal C_7)=7$,
\item $\chi_{slid}(\mathcal C_n)=5$ otherwise.
\end{itemize}

As a consequence, any graph with maximum degree $2$ has a slid-colouring with seven colours.
\end{Prop}

\begin{proof}
A colouring of $\mathcal C_n$ is a slid-colouring if and only if every four consecutive vertices have
different colours.
Then it is clear that $\chi_{slid}(\mathcal C_n)=4$ if  and only if $n\equiv 0 \bmod 4$,
$\chi_{slid}(\mathcal C_6)=6$ and $\chi_{slid}(\mathcal C_7)=7$.

If $n\equiv i \bmod 4$ ($i\neq 0$), then $\chi_{slid}(\mathcal C_n)=5$  using the colouring described by the word $(12345)^i(1234)^*$, if $n\geq 5i$ ($M^i$ means that we repeat the pattern $M$ $i$ times). It remains to consider the case $n=11$. There is no slid-colouring of
$\mathcal C_{11}$ with five colours, otherwise one colour would appear three times, and so two
occurrences of it will be at distance less than $4$. Moreover, $12345123456$ is a
slid-colouring of $\mathcal C_{11}$.

Clearly, a path has a slid-colouring with four colours and so any graph with maximum degree $2$ has a slid-colouring with seven colours.
\end{proof}

Finally, we consider the class of chordal graphs. Chordal graphs are graphs 
where each induced cycle has size at most three. They belong to the class of perfect graphs.
One of their properties (see \cite{BLS99}) is to admit a simplicial order of elimination for vertices: if $G$ is a chordal graph, there is a vertex $v$ whose neighbourhood is a clique (a {\em
simplicial vertex}), and then $G\setminus\{v\}$ is still a chordal graph. For chordal graphs, we have
$\omega(G)=\chi(G)$ where $\omega(G)$ is the {\em clique number} of $G$, i.e. the maximum size of a clique of $G$ (see \cite{BLS99}). In \cite{EGMOP10}, it
is conjectured that $\chi_{lid}(G)\leq 2\omega(G)$, for any chordal graph $G$. We give the first nontrivial bound on
$\chi_{slid}$, and so on $\chi_{lid}$, in terms of parameters $\Delta$ and $\omega$ for chordal graphs, in the direction of the previous conjecture.

\begin{Prop}
Let $G$ be a chordal graph and let $\omega=\omega(G)$. If $\omega\leq \frac{\Delta}{2}+1$, then:
$$\chi_{slid}(G) \leq 2\Delta\omega-2\omega^2+5\omega-2\Delta-2.$$

Otherwise:
$$\chi_{slid}(G) \leq \frac{\Delta(\Delta+1)}{2}+1\leq 2\omega^2-7\omega+7.$$
\end{Prop}

\begin{proof}
Let $\omega=\omega(G)$ and $\Delta$ be fixed. 
Let $M(\omega,\Delta)=\max_{1\leq d \leq \omega-1}\{d(2\Delta-2d+1)\}$. The function $d\to 2(2\Delta-2d+1)$ is maximized for $d=\tfrac{\Delta}{2}$. Using this fact, 
$M(\omega,\Delta)$ is equal to
$2\Delta\omega-2\omega^2+5\omega-2\Delta-3$ if $\omega-1\leq \tfrac{\Delta}{2}$ and to $\tfrac{\Delta(\Delta+1)}{2}$
otherwise. If $\omega-1>\tfrac{\Delta}{2}$, we clearly have $\tfrac{\Delta(\Delta+1)}{2}\leq 2\omega^2-7\omega+6$. Hence it is enough to prove that $\chi_{slid}(G) \leq M(\omega,\Delta)+1$.

We prove by induction on the number of vertices that any chordal graph with clique number at most
$\omega$ and maximum degree at most $\Delta$ has a slid-colouring with $M(\omega,\Delta)+1$ colours.
 It is clearly true for small graphs.
Let $G$ be a chordal graph with clique number at most $\omega$ and maximum degree at most $\Delta$. Let $v$ be a simplicial vertex of
$G$. By induction, let $c$ be a slid-colouring of $G\setminus\{v\}$ with $M(\omega,\Delta)+1$ colours.
Necessarily, all the vertices of $N(v)$ have different colours, and all the vertices at distance $2$ of $v$ have
colours different from colours of $N(v)$ because they are at distance at most $2$ of any vertex of $N(v)$.
Let $c'$ be the colouring of $G$ extending $c$ and giving to $v$ a completely new colour.
Then $c'$ is a slid-colouring of $G$ with $M(\omega,\Delta)+2$ colours.
Let $d\leq \omega-1$ be the degree of $v$. Then $v$ has at most $d(\Delta-d)$ vertices at distance
$2$. Observe now that $d+2d(\Delta-d)=d(2\Delta-2d+1)\leq M(\omega,\Delta)$ (by definition of $M(\omega,\Delta)$). By Lemma~\ref{lem:rec2}, we can recolour the vertex
$v$ with a colour already used and thus obtain a slid-colouring with $M(\omega,\Delta)+1$ colours.
\end{proof}


\begin{thebibliography}{1}

\bibitem{BRS03}
P.N. Balister, O.M. Riordan, and R.H. Schelp.
\newblock Vertex-distinguishing edge-colorings of graphs.
\newblock {\em J. Graph Theory}, {\bf 42}:95--109, 2003.

\bibitem{BM08} A.~Bondy and U.~S.~R.~Murty.
\newblock {\em Graph Theory}.
\newblock Springer, 3rd edition, 2008.


\bibitem{BLS99} A.~Brandst\"adt, V.B.~Le, and J.P.~Spinrad. 
\newblock {\em Graph Classes: A Survey.}
\newblock SIAM Monographs on Discrete Mathematics and
    Applications, 1999.



\bibitem{BS97}
A.C. Burris and R.H. Schelp.
\newblock Vertex-distinguishing proper edge-colorings.
\newblock {\em J. Graph Theory}, {\bf 26}:73--83, 1997.

\bibitem{CEHSZ02} G. Chartrand, D. Erwin, M.A. Henning, P.J. Slater and P. Zhang.
\newblock The locating-chromatic number of a graph.
\newblock {\em Bull. Inst. Combin. Appl.}, {\bf 36}:89-101, 2002. 

\bibitem{CHS96}
J.~Cern\'y, M.~Hor\v{n}\'ak, and R.~Sot\'ak.
\newblock Observability of a graph.
\newblock {\em Math. Slovaca}, {\bf 46}(1):21--31, 1996.

\bibitem{EGMOP10}
L.~Esperet, S.~Gravier, M.~Montassier, P.~Ochem and A.~Parreau.
\newblock Locally identifying colourings of graph.
\newblock {\em Submitted} in $2010^+$, available on arXiv at \url{http://arxiv.org/abs/1010.5624}.

\bibitem{KCL98} M.~G.~Karpovsky, K.~Chakrabarty, and L.~B.~Levitin. 
\newblock On a new class of codes for identifying vertices in graphs. \
\newblock \emph{IEEE Trans. Inform. Theory}, 44:599--611, 1998.

\bibitem{LW70} D.~R.~Lick and A.~R.~White.
\newblock $k$-degenerate graphs. 
\newblock \emph{Canad. J. Math.}, 22:1082-1096, 1970.

\bibitem{lobs}
A. Lobstein, Identifying and locating-dominating codes in graphs, a bibliography, published electronically at \url{http://perso.enst.fr/~lobstein/debutBIBidetlocdom.pdf}.

\end{thebibliography}
\end{document}